\documentclass[11pt,reqno]{amsart}
\usepackage{amsthm}
\usepackage{blkarray}
\usepackage{amsmath,amssymb,graphicx,amscd,xy}
\usepackage[usenames,dvipsnames]{color}
\usepackage{graphicx}
\usepackage{longtable}
\usepackage{mathtools}
\usepackage{color}
\usepackage{verbatim}
\usepackage{pict2e}
\usepackage[utf8,utf8x]{inputenc}
\usepackage{enumitem}
\usepackage{longtable}
\usepackage{hyperref}

\parskip=\smallskipamount

\allowdisplaybreaks

\begin{document}

\author{Lars Simon}
\address{Lars Simon, Department of Mathematical Sciences, Norwegian University of Science and Technology, Trondheim, Norway}
\email{lars.simon@ntnu.no}

\title{Homogeneous Plurisubharmonic Polynomials in Higher Dimensions}

%
\begin{abstract}
We prove several results on homogeneous plurisubharmonic polynomials on $\mathbb{C}^n$, $n\in\mathbb{Z}_{\geq 2}$. Said results are relevant to the problem of constructing local bumpings at boundary points of pseudoconvex domains of finite D'Angelo $1$-type in $\mathbb{C}^{n+1}$.
\end{abstract}

\maketitle

\section{introduction}\label{introsection}

Local bumpings at boundary points of certain bounded, smoothly bounded pseudoconvex domains of finite D'Angelo $1$-type in $\mathbb{C}^{n+1}$, $n\in\mathbb{Z}_{\geq 1}$, have been used both in the construction of peak functions (e.g.\ \cite{MR0492400}, \cite{MR1016439}, \cite{MR1207878}) and in the construction of integral kernels for solving the $\overline{\partial}$-equation (e.g.\ \cite{MR835766}, \cite{MR1070924}).

As explained in \cite{MR2452636}, \cite{MR2993440}, the problem of constructing such local bumpings naturally leads to the study of homogeneous plurisubharmonic polynomials on $\mathbb{C}^n$. Furthermore, in \cite{1909.04080}, results on homogeneous plurisubharmonic polynomials on $\mathbb{C}^2$ by Bharali, Stensønes \cite{MR2452636}, applied in combination with results from \cite{FornStens2010}, played an important role in establishing sup-norm estimates for solutions to the $\overline{\partial}$-equation for a large class of pseudoconvex domains in $\mathbb{C}^3$. Specifically, the crucial results on homogeneous plurisubharmonic polynomials on $\mathbb{C}^2$ are the following:

\theoremstyle{plain}
\newtheorem*{resulteinsmaerz}{Result}
\begin{resulteinsmaerz}[{\cite[Proposition 1]{MR2452636}}]
	\label{resulteinsmaerz}
	Let $P$ be a homogeneous, plurisubharmonic, non-pluriharmonic polynomial on $\mathbb{C}^2$. Then there are at most finitely many complex lines through the origin in $\mathbb{C}^2$ along which $P$ is harmonic.
\end{resulteinsmaerz}

\theoremstyle{plain}
\newtheorem*{resultzweimaerz}{Result}
\begin{resultzweimaerz}[{\cite[Theorem 2]{MR2452636}}]
	\label{resultzweimaerz}
	Let $P\colon\mathbb{C}^2\to\mathbb{R}$ be a non-constant, homogeneous, plurisubharmonic polynomial without pluriharmonic terms. Assume that there exists a holomorphic function $g\colon\mathbb{C}^2\to\mathbb{C}$, nonsingular on a non-empty open set $U\subseteq\mathbb{C}^2$, such that $P$ is harmonic along every level set of $g|_U$.\\
	Then there exist a homogeneous, subharmonic polynomial $s\colon\mathbb{C}\to\mathbb{R}$ and a homogeneous holomorphic polynomial $h\colon\mathbb{C}^2\to\mathbb{C}$, such that $P=s\circ h$ on $\mathbb{C}^2$.
\end{resultzweimaerz}

\theoremstyle{plain}
\newtheorem*{resultdreimaerz}{Result}
\begin{resultdreimaerz}[{\cite[Theorem 3]{MR2452636}}]
	\label{resultdreimaerz}
	Let $P\colon\mathbb{C}^2\to\mathbb{R}$ be a non-constant, homogeneous, plurisubharmonic polynomial without pluriharmonic terms and assume that $P$ is homogeneous of degree $2d_1$ in $z_1, \overline{z_1}$ and homogeneous of degree $2d_2$ in $z_2, \overline{z_2}$, where $d_1 , d_2 \in\mathbb{Z}_{\geq 1}$.
	Then there exist a homogeneous, subharmonic polynomial $s\colon\mathbb{C}\to\mathbb{R}$ without harmonic terms and integers $j,l\in\mathbb{Z}_{\geq 1}$, such that
	\begin{align*}
	P(z_1 , z_2 ) = s({z_1}^j {z_2}^l )\text{ for all }(z_1 , z_2 )\in\mathbb{C}^2\text{.}
	\end{align*}
\end{resultdreimaerz}

When attempting to adapt the methods from \cite{1909.04080} to higher dimensions, it is natural to ask for generalizations of the above-mentioned results on homogeneous plurisubharmonic polynomials on $\mathbb{C}^2$ to higher dimensions. Specifically, it is natural to ask the following questions:

\theoremstyle{plain}
\newtheorem*{questioneins}{Question A}
\begin{questioneins}
	\label{questioneins}
	Given a homogeneous, plurisubharmonic, non-pluriharmonic polynomial $P$ on $\mathbb{C}^n$, $n\geq 2$, is it true that there are at most finitely many complex hyperplanes through the origin in $\mathbb{C}^n$ along which $P$ is pluriharmonic?
\end{questioneins}

\theoremstyle{plain}
\newtheorem*{questionzwei}{Question B}
\begin{questionzwei}
	\label{questionzwei}
	Let $P\colon\mathbb{C}^n\to\mathbb{R}$, $n\in\mathbb{Z}_{\geq 2}$, be a non-constant, homogeneous, plurisubharmonic polynomial without pluriharmonic terms. Assume that there exists a holomorphic map $G\colon\mathbb{C}^n\to\mathbb{C}^m$, $1\leq m\leq n-1$, nonsingular on a non-empty open set $U\subseteq\mathbb{C}^n$, such that $P$ is pluriharmonic along every level set of $G|_U$.\\
	Are there necessarily a homogeneous, plurisubharmonic polynomial $Q\colon\mathbb{C}^{m}\to\mathbb{R}$ and holomorphic polynomials $F_1, \dots ,F_m\colon\mathbb{C}^n\to\mathbb{C}$, all homogeneous of the same degree, such that $P=Q\circ (F_1 ,\dots ,F_m )$ on $\mathbb{C}^n$?
\end{questionzwei}

\theoremstyle{plain}
\newtheorem*{questiondrei}{Question C}
\begin{questiondrei}
	\label{questiondrei}
	Let $P\colon\mathbb{C}^n\to\mathbb{R}$, $n\in\mathbb{Z}_{\geq 2}$, be a non-constant, homogeneous, plurisubharmonic polynomial without pluriharmonic terms and assume $P$ is homogeneous in $l$ variables separately, $1\leq l\leq n-1$.\\
	Are there necessarily a homogeneous, plurisubharmonic polynomial $Q\colon\mathbb{C}^{n-l}\to\mathbb{R}$ and holomorphic polynomials $F_1, \dots ,F_{n-l}\colon\mathbb{C}^n\to\mathbb{C}$, all homogeneous of the same degree, such that $P=Q\circ (F_1 ,\dots ,F_{n-l} )$ on $\mathbb{C}^n$?
\end{questiondrei}

The purpose of this paper is to provide a detailed answer to Questions A, B and C. A formal statement of the results can be found in Section \ref{statementofresultssection}.

The answer to Question A is ``yes'' (Proposition \ref{theoremfinmanyhyperplanes}).

 The answer to Question B is ``no'' in general, even if we additionally assume that the component functions of $G$ are holomorphic polynomials which are all homogeneous of the same degree (Proposition \ref{propositioncounterexample}). However, the answer is ``yes'' in the special case $m=1$, i.e., when the polynomial is pluriharmonic along the level sets of a single holomorphic function (Theorem \ref{theoremfolisinglefct}).
 
The answer to Question C is ``no'' in general (Proposition \ref{propositioncounterexample}). However, in the special case where the polynomial is homogeneous in all $n$ variables separately, the answer is ``yes'' (Corollary \ref{corollaryhomogeveryvarseparately}). Furthermore, we get that the answer is ``almost yes'', or ``yes, up to certain singular holomorphic coordinate changes'' in the general setting (Theorem \ref{theoremhomogseparately}). In many cases, the latter theorem can be used to get the desired bumping results. This is important, since Proposition \ref{propositioncounterexample} shows that the result by Bharali, Stensønes \cite[Theorem 3]{MR2452636} does not generalize in this setting.

\section{Statement of Results}\label{statementofresultssection}

In this section we state the results of this paper. All the proofs can be found in the later sections.

The answer to Question A is ``yes''. We have:

\theoremstyle{plain}
\newtheorem{theoremfinmanyhyperplanes}[propo]{Proposition}
\begin{theoremfinmanyhyperplanes}
	\label{theoremfinmanyhyperplanes}
	Let $P\colon\mathbb{C}^n\to\mathbb{R}$, $n\in\mathbb{Z}_{\geq 2}$, be a homogeneous, plurisubharmonic, non-pluriharmonic polynomial. Then there are at most finitely many complex hyperplanes through the origin in $\mathbb{C}^n$ along which $P$ is pluriharmonic.
\end{theoremfinmanyhyperplanes}

The answer to Question B is ``no'', even if we additionally assume that the component functions of $G$ are holomorphic polynomials which are all homogeneous of the same degree. The answer to Question C is ``no'' as well. All of this is implied by the following:

\theoremstyle{plain}
\newtheorem{propositioncounterexample}[propo]{Proposition}
\begin{propositioncounterexample}
	\label{propositioncounterexample}
	Let $P\colon\mathbb{C}^3\to\mathbb{R}$,
	\begin{align*}
	P(z,w_1,w_2)=|z|^2\cdot (|w_1 |^4+|w_1^2-w_1 w_2|^2+|w_2|^4)\text{.}
	\end{align*}
	Then $P$ is a non-constant, homogeneous, plurisubharmonic polynomial without pluriharmonic terms and $P$ is homogeneous in one variable separately (see Definition \ref{notatiohomogseperately}). Furthermore, away from the coordinate hyperplanes, $P$ is pluriharmonic along the level sets of $G\colon\mathbb{C}^3\to\mathbb{C}^2$, $G(z,w_1 ,w_2 )=(zw_1^2,zw_2^2)$.\\
	However, there do not exist a homogeneous, plurisubharmonic polynomial $Q\colon\mathbb{C}^2\to\mathbb{R}$ and holomorphic polynomials $F_1 ,F_2\colon\mathbb{C}^3\to\mathbb{C}$, homogeneous of the same degree, such that $P=Q\circ (F_1 , F_2 )$ on $\mathbb{C}^3$.
\end{propositioncounterexample}

Nevertheless, the answer to Question B in the special case where $m=1$ is ``yes'':

\theoremstyle{plain}
\newtheorem{theoremfolisinglefct}[propo]{Theorem}
\begin{theoremfolisinglefct}
	\label{theoremfolisinglefct}
	Let $P\colon\mathbb{C}^n\to\mathbb{R}$, $n\in\mathbb{Z}_{\geq 2}$, be a non-constant, homogeneous, plurisubharmonic polynomial without pluriharmonic terms. Assume that there exists a holomorphic function $G\colon\mathbb{C}^n\to\mathbb{C}$, nonsingular on a non-empty open set $U\subseteq\mathbb{C}^n$, such that $P$ is pluriharmonic along every level set of $G|_U$.\\
	Then there exist a homogeneous, subharmonic polynomial $s\colon\mathbb{C}\to\mathbb{R}$ without harmonic terms and a homogeneous holomorphic polynomial $h\colon\mathbb{C}^n\to\mathbb{C}$, such that $P=s\circ h$ on $\mathbb{C}^n$.
\end{theoremfolisinglefct}

Note that Theorem \ref{theoremfolisinglefct} generalizes the upper mentioned result by Bharali, Stensønes \cite[Theorem 2]{MR2452636} to higher dimension.

Although the answer to Question C is ``no'', we get the following result, which says that the answer is ``yes, up to singular holomorphic coordinate changes'':

\theoremstyle{plain}
\newtheorem{theoremhomogseparately}[propo]{Theorem}
\begin{theoremhomogseparately}
	\label{theoremhomogseparately}
	Let $P\colon\mathbb{C}^n\to\mathbb{R}$, $n\in\mathbb{Z}_{\geq 2}$, be a non-constant plurisubharmonic polynomial without pluriharmonic terms and assume that $P$ is homogeneous of degree $2k$, $k\in\mathbb{Z}_{\geq 1}$. Let $1\leq l\leq n-1$ and assume that $P$ is homogeneous of degree $2d_j$, $d_j\in\mathbb{Z}_{>0}$, in $z_j$, $\overline{z_j}$ for $j=1,\dots ,l$ (see Def.\ \ref{notatiohomogseperately}). Assume furthermore that $k-D>0$, where $D=d_1 +\dots +d_l$. Write $d:=\gcd (d_1 ,\dots ,d_l ,k)\in\mathbb{Z}_{\geq 1}$.\\
	Then there exists a plurisubharmonic polynomial $Q\colon\mathbb{C}^{n-l}\to\mathbb{R}$ without pluriharmonic terms, homogeneous of degree $2k-2D$, with the property that both the holomorphic and the anti-holomorphic degree (see Section \ref{prelimsection}) of every term appearing in $Q$ are divisible by the integer $(k-D)/d$, such that we have for all $(z_1 ,\dots ,z_n)\in\mathbb{C}^n$:
	\begin{align*}
	P(z_1 ,\dots ,z_n )=Q(\tau z_{l+1},\dots ,\tau z_n )\text{,}
	\end{align*}
	for every solution $\tau\in\mathbb{C}$ of $\tau^{(k-D)/d}=z_1^{d_1/d}\cdots z_l^{d_l/d}$. 
\end{theoremhomogseparately}

Alternatively, we can carry out a {\emph{singular}} holomorphic coordinate change $\Phi\colon\mathbb{C}^n\to\mathbb{C}^n$, $(z_1 ,\dots ,z_n )\mapsto ({z_1}^{(k-D)/d},\dots ,{z_l}^{(k-D)/d},z_{l+1},\dots ,z_n )$ and write
\begin{align*}
(P\circ\Phi )(z_1 ,\dots ,z_n )=Q(z_1^{d_1/d}\cdots z_l^{d_l/d}z_{l+1},\dots ,z_1^{d_1/d}\cdots z_l^{d_l/d}z_n )\text{.}
\end{align*}
Note that, without the assumptions $d_1 ,\dots ,d_l ,k-D>0$ in Theorem \ref{theoremhomogseparately}, $P$ is effectively a polynomial in fewer than $n$ variables, hence we can ignore that case.

As a corollary (of the proof) of Theorem \ref{theoremhomogseparately} we get that the answer to Question C in the special case where $P$ is homogeneous in all $n$ variables separately is ``yes'':

\theoremstyle{plain}
\newtheorem{corollaryhomogeveryvarseparately}[propo]{Corollary}
\begin{corollaryhomogeveryvarseparately}
	\label{corollaryhomogeveryvarseparately}
	Let $P\colon\mathbb{C}^n\to\mathbb{R}$, $n\in\mathbb{Z}_{\geq 2}$, be a non-constant, homogeneous, plurisubharmonic polynomial without pluriharmonic terms and assume that $P$ is homogeneous of degree $2d_j$, $d_j\in\mathbb{Z}_{>0}$, in $z_j$, $\overline{z_j}$ for $j=1,\dots ,n$ (see Def.\ \ref{notatiohomogseperately}). Write $d:=\gcd (d_1 ,\dots ,d_n )\in\mathbb{Z}_{\geq 1}$.\\
	Then there exists a homogeneous, subharmonic polynomial $s\colon\mathbb{C}\to\mathbb{R}$ without harmonic terms, such that
	\begin{align*}
	P(z_1 ,\dots ,z_n )=s({z_1}^{d_1 /d}\cdots {z_n}^{d_n /d})
	\end{align*}
	for all $(z_1 ,\dots ,z_n )\in\mathbb{C}^n$.
\end{corollaryhomogeveryvarseparately}

Note that Corollary \ref{corollaryhomogeveryvarseparately} generalizes the upper mentioned result by Bharali, Stensønes \cite[Theorem 3]{MR2452636} to higher dimension.

\section{Preliminaries}\label{prelimsection}
For the remainder of this section we fix an integer $n\geq 2$ and a non-constant polynomial $P\colon\mathbb{C}^n\to\mathbb{R}$ with the following properties:
\begin{itemize}
	\item{$P$ is $\mathbb{R}$-homogeneous of degree $2k$, for some positive integer $k$,}
	\item{$P$ is plurisubharmonic,}
	\item{$P$ does not have any pluriharmonic terms (i.e., purely holomorphic or purely anti-holomorphic terms).}
\end{itemize}
In particular, there exists a collection $(a_{{\alpha},{\beta}})_{({\alpha},{\beta})\in\mathcal{J}}$ of complex numbers where
\begin{itemize}
	\item $\mathcal{J}$ is the set of all pairs $({\alpha},{\beta})\in (\mathbb{Z}_{\geq 0})^n\times (\mathbb{Z}_{\geq 0})^n$ satisfying $|\alpha |>0$, $|\beta |>0$ and $|\alpha |+|\beta |=2k$,
	\item $\overline{a_{{\alpha},{\beta}}}=a_{{\beta},{\alpha}}$ for all $({\alpha},{\beta})\in\mathcal{J}$,
\end{itemize}
such that
\begin{align*}
P(z)=\sum_{({\alpha},{\beta})\in\mathcal{J}} a_{\alpha ,\beta}z^\alpha\overline{z}^\beta
\end{align*}
for all $z=(z_1 , \dots , z_n )\in\mathbb{C}^n$. Here we are making use of the usual multi-index notation: $|\alpha |=\alpha_1 +\dots +\alpha_n$ and $z^\alpha =z_1^{\alpha_1}\cdots z_n^{\alpha_n}$ (and analogously for $\beta$ and $\overline{z}^\beta$). If $a_{\alpha ,\beta}\neq 0$, then we say that $|\alpha |$ (resp.\ $|\beta |$) is the holomorphic (resp.\ anti-holomorphic) degree of the term $a_{\alpha ,\beta}z^\alpha\overline{z}^\beta$.

Furthermore, let $\mathcal{L}(P;p,V)$ denote the Levi form of $P$ at the point $p\in\mathbb{C}^n$ in direction $V=(V_1 ,\dots ,V_n )^t\in\mathbb{C}^n$, i.e.,
\begin{align*}
\mathcal{L}(P;p,V)=
(V_1,\dots ,V_n )
\begin{pmatrix}
\frac{\partial^2 P}{\partial{z_1}\partial\overline{z_1}}(p) & \dots & \frac{\partial^2 P}{\partial{z_1}\partial\overline{z_n}}(p) \\
\vdots & \ddots & \vdots \\
\frac{\partial^2 P}{\partial{z_n}\partial\overline{z_1}}(p) & \dots & \frac{\partial^2 P}{\partial{z_n}\partial\overline{z_n}}(p)
\end{pmatrix}
\begin{pmatrix}
\overline{V_1} \\
\vdots \\
\overline{V_n}
\end{pmatrix}\text{.}
\end{align*}

\theoremstyle{plain}
\newtheorem{lemmalevizeroaveraging}[propo]{Lemma}
\begin{lemmalevizeroaveraging}
\label{lemmalevizeroaveraging}
Let $\mathcal{A}=\{\alpha\in (\mathbb{Z}_{\geq 0})^n\colon |\alpha |=k\text{ and }a_{\alpha ,\alpha}\neq 0\}$ and let
\begin{align*}
\mathcal{C}=\left\{(c_1 ,\dots ,c_n )\in\mathbb{C}^n\colon \sum_{j=1}^{n}\alpha_j c_j =0\text{ for all }\alpha\in\mathcal{A}\right\}\text{.}
\end{align*}
Then we have for all $(c_1 ,\dots ,c_n )\in\mathcal{C}$ and for all $z=(z_1 ,\dots ,z_n )\in\mathbb{C}^n$:
\begin{align*}
\mathcal{L}(P;z,(c_1 z_1 ,\dots ,c_n z_n )^t)=0\text{.}
\end{align*}
\end{lemmalevizeroaveraging}

\begin{proof}
For all $z=(z_1 ,\dots ,z_n )\in\mathbb{C}^n$, $(c_1 ,\dots ,c_n )\in\mathbb{C}^n$ a straightforward calculation shows that
\begin{align*}
\mathcal{L}(P;z,(c_1 z_1 ,\dots ,c_n z_n )^t)
=\sum_{(\alpha ,\beta )\in\mathcal{J}}a_{\alpha ,\beta}\cdot \Bigg({\sum_{j=1}^{n} \alpha_j c_j}\Bigg)\cdot\overline{\Bigg({\sum_{j=1}^{n}\beta_j c_j}\Bigg)}\cdot z^\alpha \overline{z}^\beta\text{.}
\end{align*}
Assume for the sake of a contradiction that the claim is wrong. We then find some $(c_1 ,\dots ,c_n )\in\mathcal{C}$ and some $r_1 ,\dots ,r_n\in\mathbb{R}_{\geq 0}$, $\phi_1 ,\dots ,\phi_n\in [0,2\pi )$, such that:
\begin{align*}
\mathcal{L}(P;(r_1 e^{i\phi_1},\dots ,r_n e^{i\phi_n}),(c_1 r_1 e^{i\phi_1},\dots ,c_n r_n e^{i\phi_n})^t)\neq 0\text{.}
\end{align*}
By continuity, and since $P$ is plurisubharmonic, we then get:
\begin{align*}
0 & <\int_{0}^{2\pi}\dots\int_{0}^{2\pi}\mathcal{L}(P;(r_1 e^{i\theta_1},\dots ,r_n e^{i\theta_n}),(c_1 r_1 e^{i\theta_1},\dots ,c_n r_n e^{i\theta_n})^t) d\theta_1\dots d\theta_n\\
& =\sum_{(\alpha ,\beta )\in\mathcal{J}}a_{\alpha ,\beta}\cdot \Bigg({\sum_{j=1}^{n} \alpha_j c_j}\Bigg)\cdot\overline{\Bigg({\sum_{j=1}^{n}\beta_j c_j}\Bigg)}\cdot {r_1}^{\alpha_1 +\beta_1}\cdots{r_n}^{\alpha_n +\beta_n}\\
& \phantom{=\sum_{(\alpha ,\beta )\in\mathcal{J}}a_{\alpha ,\beta}}\cdot \left(\int_{0}^{2\pi} e^{i(\alpha_1 -\beta_1 )\theta_1} d\theta_1\right)\cdots\left(\int_{0}^{2\pi} e^{i(\alpha_n -\beta_n )\theta_n} d\theta_n\right)\\
& =(2\pi )^n\sum_{\alpha\in\mathcal{A}}a_{\alpha ,\alpha}\cdot \Bigg|{\sum_{j=1}^{n} \alpha_j c_j}\Bigg|^2 \cdot {r_1}^{\alpha_1 +\alpha_1}\cdots{r_n}^{\alpha_n +\alpha_n}\\
& =0\text{,}
\end{align*}
where the last equality is due to the fact that $(c_1 ,\dots ,c_n )\in\mathcal{C}$. We have arrived at the desired contradiction; the claim follows.
\end{proof}

For all $\beta\in (\mathbb{Z}_{\geq 0})^n$ with $1\leq |\beta |\leq 2k-1$ we define a homogeneous holomorphic polynomial $P_{\beta}\colon\mathbb{C}^n\to\mathbb{C}$,
\begin{align*}
P_{\beta}(z)=\sum_{\alpha\colon |\alpha |=2k-|\beta |}a_{\alpha ,\beta}z^{\alpha}\text{.}
\end{align*}
In particular we can write
\begin{align*}
P(z)=\sum_{\beta\colon 1\leq |\beta |\leq 2k-1}{\overline{z}^{\beta}P_{\beta}(z)}\text{.}
\end{align*}

\theoremstyle{plain}
\newtheorem{lemmaplurihalonglevelsets}[propo]{Lemma}
\begin{lemmaplurihalonglevelsets}
	\label{lemmaplurihalonglevelsets}
	Assume that there exists a holomorphic map $G\colon\mathbb{C}^n\to\mathbb{C}^m$, $1\leq m\leq n-1$, nonsingular on a non-empty open set $U\subseteq\mathbb{C}^n$, such that $P$ is pluriharmonic along every level set of $G|_U$.\\
	Then, for all $i_1 ,\dots ,i_m,L\in\{1,\dots ,n\} $ (not necessarily pairwise distinct) and for all $\beta\in (\mathbb{Z}_{\geq 0})^n$ with $1\leq |\beta |\leq 2k-1$, the following equality holds on $\mathbb{C}^n$:
	\begin{align*}
	\det
	\begin{pmatrix}
	\frac{\partial G_1}{\partial z_{i_1}} & \dots  & \frac{\partial G_1}{\partial z_{i_m}} \\
	\vdots &  \ddots & \vdots \\
	\frac{\partial G_m}{\partial z_{i_1}} & \dots  & \frac{\partial G_m}{\partial z_{i_m}}
	\end{pmatrix}\cdot \frac{\partial P_{\beta}}{\partial z_L}
	=\sum_{j=1}^{m}\det
	\begin{pmatrix}
	\frac{\partial G_1}{\partial z_{i_1}} & \dots  & \frac{\partial G_1}{\partial z_{i_m}} \\
	\vdots &  \ddots & \vdots \\
	\frac{\partial G_{j-1}}{\partial z_{i_1}} & \dots  & \frac{\partial G_{j-1}}{\partial z_{i_m}} \\
	\frac{\partial P_{\beta}}{\partial z_{i_1}} & \dots  & \frac{\partial P_{\beta}}{\partial z_{i_m}} \\
	\frac{\partial G_{j+1}}{\partial z_{i_1}} & \dots  & \frac{\partial G_{j+1}}{\partial z_{i_m}} \\
	\vdots &  \ddots & \vdots \\
	\frac{\partial G_m}{\partial z_{i_1}} & \dots  & \frac{\partial G_m}{\partial z_{i_m}}
	\end{pmatrix}\cdot \frac{\partial G_j}{\partial z_L}\text{.}
	\end{align*}
\end{lemmaplurihalonglevelsets}

\begin{proof}
Fix a point $p\in U$. Then there exist an open neighborhood $W\subseteq U$ of $p$ and holomorphic maps $K_1 ,\dots ,K_{n-m}\colon W\to\mathbb{C}^n\setminus\{0\}$, such that $\{K_1 (z),\dots ,K_{n-m}(z)\}$ is a basis for the null space of $G'(z)\in\mathbb{C}^{m\times n}$ for all $z\in W$. For $l=1,\dots ,n-m$, denote the component functions of $K_l$ as ${K_l}^{(1)},\dots ,{K_l}^{(n)}$. Since $P$ is pluriharmonic along every level set of $G|_U$, we get
\begin{align*}
0=({K_l}^{(1)},\dots ,{K_l}^{(n)})
\begin{pmatrix}
\frac{\partial^2 P}{\partial{z_1}\partial\overline{z_1}} & \dots & \frac{\partial^2 P}{\partial{z_1}\partial\overline{z_n}} \\
\vdots & \ddots & \vdots \\
\frac{\partial^2 P}{\partial{z_n}\partial\overline{z_1}} & \dots & \frac{\partial^2 P}{\partial{z_n}\partial\overline{z_n}}
\end{pmatrix}
\begin{pmatrix}
\overline{{K_l}^{(1)}} \\
\vdots \\
\overline{{K_l}^{(n)}}
\end{pmatrix}
\end{align*}
on $W$ for $l=1,\dots ,n-m$; but since the Complex Hessian matrix of $P$ is positive semidefinite, we even get
\begin{align*}
(0,\dots ,0)=({K_l}^{(1)},\dots ,{K_l}^{(n)})
\begin{pmatrix}
\frac{\partial^2 P}{\partial{z_1}\partial\overline{z_1}} & \dots & \frac{\partial^2 P}{\partial{z_1}\partial\overline{z_n}} \\
\vdots & \ddots & \vdots \\
\frac{\partial^2 P}{\partial{z_n}\partial\overline{z_1}} & \dots & \frac{\partial^2 P}{\partial{z_n}\partial\overline{z_n}}
\end{pmatrix}\text{.}
\end{align*}
Writing
\begin{align*}
P(z)=\sum_{\beta\colon 1\leq |\beta |\leq 2k-1}{\overline{z}^{\beta}P_{\beta}(z)}\text{,}
\end{align*}
as above, we get for $I=1,\dots ,n$:
\begin{align*}
0 & = \sum_{J=1}^{n}{K_l}^{(J)}(z)\cdot\frac{{\partial}^2 P}{\partial z_J\partial\overline{z_I}}(z)\\
& =\sum_{\beta\colon 1\leq |\beta |\leq 2k-1}{\left(\beta_I\overline{z_1}^{\beta_1}\cdots\overline{z_I}^{\beta_I -1}\cdots\overline{z_n}^{\beta_n}\cdot\sum_{J=1}^{n}{K_l}^{(J)}(z)\frac{\partial P_{\beta}}{\partial z_J}(z)\right)}
\end{align*}
for all $z\in W$, $l=1,\dots ,n-m$; hence
\begin{align*}
0 & =\sum_{I=1}^{n}\overline{z_I}\cdot 0\\
& =\sum_{\beta\colon 1\leq |\beta |\leq 2k-1}{\left( |\beta |\cdot\overline{z}^{\beta}\cdot\sum_{J=1}^{n}{K_l}^{(J)}(z)\frac{\partial P_{\beta}}{\partial z_J}(z)\right)}\text{.}
\end{align*}
Owing to the fact that the ${K_l}^{(J)}$ and the $P_\beta$ are holomorphic, we then get for all $\beta\in ({\mathbb{Z}_{\geq 0}})^n$ with $1\leq |\beta |\leq 2k-1$:
\begin{align*}
0=\sum_{J=1}^{n}{K_l}^{(J)}(z)\frac{\partial P_{\beta}}{\partial z_J}(z)\text{ for all }z\in W, l\in\{1,\dots ,n-m\}\text{,}
\end{align*}
i.e., $K_l(z)=({K_l}^{(1)}(z),\dots ,{K_l}^{(n)}(z))^t$ is in the null space of the matrix $P_{\beta}'(z)\in\mathbb{C}^{1\times n}$. But this implies that the null space of the matrix
\begin{align*}
\begin{pmatrix}
\frac{\partial G_1}{\partial z_1}(z) & \dots  & \frac{\partial G_1}{\partial z_n}(z) \\
\vdots &  \ddots & \vdots \\
\frac{\partial G_m}{\partial z_1}(z) & \dots  & \frac{\partial G_m}{\partial z_n}(z) \\
\frac{\partial P_\beta}{\partial z_1}(z) & \dots  & \frac{\partial P_\beta}{\partial z_n}(z)
\end{pmatrix}\in\mathbb{C}^{(m+1)\times n}
,z\in W\text{,}
\end{align*}
is $(n-m)$-dimensional for all $\beta$ with $1\leq |\beta |\leq 2k-1$, so the rank of said matrix is $m$. Hence, given $i_1 ,\dots ,i_m,L\in\{1,\dots ,n\}$, we have
\begin{align*}
0=\det
\begin{pmatrix}
\frac{\partial G_1}{\partial z_{i_1}} & \dots  & \frac{\partial G_1}{\partial z_{i_m}} & \frac{\partial G_1}{\partial z_L} \\
\vdots &  \ddots & \vdots & \vdots \\
\frac{\partial G_m}{\partial z_{i_1}} & \dots  & \frac{\partial G_m}{\partial z_{i_m}} & \frac{\partial G_m}{\partial z_L} \\
\frac{\partial P_\beta}{\partial z_{i_1}} & \dots  & \frac{\partial P_\beta}{\partial z_{i_m}} & \frac{\partial P_\beta}{\partial z_L}
\end{pmatrix}
\end{align*}
on $W$ for all $\beta$ with $1\leq |\beta |\leq 2k-1$; noting that all the entries of the latter matrix are holomorphic on $\mathbb{C}^n$, the identity theorem gives that the determinant vanishes on all of $\mathbb{C}^n$. The claim follows by Laplace expanding by the last column and calculating.
\end{proof}

\theoremstyle{definition}
\newtheorem{notatiohomogseperately}[propo]{Definition}
\begin{notatiohomogseperately}
	\label{notatiohomogseperately}
	Given $l\in\{1,\dots ,n\}$, we say that $P$ is {\emph{homogeneous in}} $l$ {\emph{variables separately}}, provided there exist integers $1\leq i_1 <\dots <i_l\leq n$ and integers $d_{i_1},\dots ,d_{i_l}\geq 1$, such that for every $({\alpha},{\beta})\in\mathcal{J}$ with ${a_{{\alpha},{\beta}}}\neq 0$ we have $\alpha_{i_1}+\beta_{i_1}=2d_{i_1},\dots , \alpha_{i_l}+\beta_{i_l}=2d_{i_l}$. In this case we say that $P$ is homogeneous of degree $2d_{i_j}$ in $z_{i_j},\overline{z_{i_j}}$ for all $j\in\{1,\dots ,l\}$.
\end{notatiohomogseperately}

\theoremstyle{remark}
\newtheorem{notehomogallvarsep}[propo]{Note}
\begin{notehomogallvarsep}
	\label{notehomogallvarsep}
	We restrict attention to even degrees in Definition \ref{notatiohomogseperately} due to the plurisubharmonicity requirement. Note furthermore that, in the case $l=n-1$, the polynomial $P$ is necessarily homogeneous in all $n$ variables separately, since $P$ is homogeneous.
\end{notehomogallvarsep}

\theoremstyle{plain}
\newtheorem{lemmahomogsepidentifyholmap}[propo]{Lemma}
\begin{lemmahomogsepidentifyholmap}
\label{lemmahomogsepidentifyholmap}
Let $1\leq l\leq n-1$ and assume that $P$ is homogeneous of degree $2d_j$, $d_j\in\mathbb{Z}_{>0}$, in $z_j$, $\overline{z_j}$ for $j=1,\dots ,l$. Assume furthermore that $k-D>0$, where $D=d_1 +\dots +d_l$.\\
Then, away from the coordinate hyperplanes, $P$ is pluriharmonic along the level sets of $G\colon\mathbb{C}^n\to\mathbb{C}^{n-l}$,
\begin{align*}
G(z_1 ,\dots ,z_n )={z_1}^{d_1}\cdots{z_l}^{d_l}\cdot \left({z_{l+1}}^{k-D},{z_{l+2}}^{k-D},\dots ,{z_n}^{k-D}\right)\text{.}
\end{align*}
\end{lemmahomogsepidentifyholmap}

\begin{proof}
$G$ is nonsingular on $U:=\{(z_1 ,\dots ,z_n )\in\mathbb{C}^n\colon z_1\neq 0,\dots ,z_n\neq 0 \}$. Given $z\in U$, we have to show that $\mathcal{L}(P;z,V)=0$ for all $V$ in the null space of $G'(z)\in\mathbb{C}^{(n-l)\times n}$. But since the Complex Hessian matrix of $P$ is positive semidefinite, it suffices to verify this for a {\emph{basis}} of said null space. If $z\in U$, then the collection of vectors\begin{align*}
\begin{blockarray}{cc}
\begin{block}{(c)c}
0 &  \\
\vdots &  \\
0 &  \\
(k-D)z_j & j\text{-th entry} \\
0 &  \\
\vdots &  \\
0 &  \\
-d_j z_{l+1} & (l+1)\text{-th entry} \\
\vdots &  \\
-d_j z_{n} &  \\
\end{block}
\end{blockarray}\text{, where }j\in\{1,\dots ,l\}
\end{align*}
forms a basis for the null space of $G'(z)\in\mathbb{C}^{(n-l)\times n}$. Hence, with $\mathcal{A}$ and $\mathcal{C}$ as in Lemma \ref{lemmalevizeroaveraging}, it suffices to show that $C_j\in\mathcal{C}$ for $j\in\{1,\dots ,l\}$, where
\begin{align*}
\begin{blockarray}{cccccccccccc}
\begin{block}{c(cccccccccc)c}
C_j:= & 0, & \dots , & 0, & k-D, & 0, & \dots , & 0, & -d_j , & \dots , & -d_j & \in\mathbb{C}^n\text{.}\\
\end{block}
\begin{block}{cccccccccccc}
& & & & j\text{-th} & & & & (l+1)\text{-th} & & &\\
\end{block}
\end{blockarray}
\end{align*}
To this end we consider some $j\in\{1,\dots ,l\}$ and some $\alpha\in\mathcal{A}$. Since $a_{\alpha ,\alpha}\neq 0$ and by assumption on $P$ we then have $\alpha_1 =d_1 ,\dots ,\alpha_l =d_l$, and $\alpha_{l+1}+\dots +\alpha_n =k-D$. Writing $C_j=:({c_1}^{(j)} ,\dots ,{c_n}^{(j)} )$, we then have:
\begin{align*}
\sum_{s=1}^{n} \alpha_s {c_s}^{(j)} =(k-D)\alpha_j -d_j (\alpha_{l+1}+\dots +\alpha_n )=0\text{,}
\end{align*}
as desired.
\end{proof}

\section{Proof of Proposition \ref{theoremfinmanyhyperplanes}}\label{hyperplanesection}

\begin{proof}[Proof of Proposition \ref{theoremfinmanyhyperplanes}]
We proceed by induction on the dimension $n$. The $2$-dimensional case was handled by Bharali, Stensønes \cite{MR2452636}, so let $n\geq 3$ and assume the claim holds in dimensions $2,\dots ,n-1$. Let $P$ be as in the statement of the theorem and assume for the sake of a contradiction that $P$ is pluriharmonic along infinitely many complex hyperplanes through $0\in\mathbb{C}^n$. We then find a sequence $(\mathcal{H}_j )_{j\in\mathbb{Z}_{\geq 1}}$ of {\emph{pairwise distinct}} such hyperplanes. It is furthermore easy to see that there exists a complex hyperplane $A$ through $0\in\mathbb{C}^n$, such that $P$ is {\emph{not}} pluriharmonic along $A$. Since $P$ is pluriharmonic along each $\mathcal{H}_j$, $j\in\mathbb{Z}_{\geq 1}$, we get that $P$ is pluriharmonic along $A\cap\mathcal{H}_j$ for all $j\in\mathbb{Z}_{\geq 1}$. Hence, by induction, the set $\{A\cap\mathcal{H}_j \colon j\in\mathbb{Z}_{\geq 1} \}$ is finite. So, there exists an $(n-2)$-dimensional complex vector subspace $V$ of $A$, such that $A\cap\mathcal{H}_j = V$ for {\emph{infinitely many}} $j\in\mathbb{Z}_{\geq 1}$. Thus, after deleting some members of the sequence if necessary, we can assume that
\begin{align*}
A\cap\mathcal{H}_j = V\text{ for all }j\in\mathbb{Z}_{\geq 1}\text{.}
\end{align*}
It is easy to verify that there exists a complex hyperplane $B$ through $0\in\mathbb{C}^n$, such that $P$ is {\emph{not}} pluriharmonic along $B$ and $B$ does not contain $V$. By repeating the same argument and again deleting some members of the sequence if necessary, we find an $(n-2)$-dimensional complex vector subspace $W$ of $B$, such that
\begin{align*}
B\cap\mathcal{H}_j = W\text{ for all }j\in\mathbb{Z}_{\geq 1}\text{.}
\end{align*}
Hence every $\mathcal{H}_j$ contains $V+W$. However, $B$ contains $W$ but does not contain $V$, so we get that $V+W$ is at least $(n-1)$-dimensional. We conclude that $\mathcal{H}_j = V+W$ for all $j\in\mathbb{Z}_{\geq 1}$. Since the members of the sequence $(\mathcal{H}_j )_{j\in\mathbb{Z}_{\geq 1}}$ were chosen to be pairwise distinct, we have arrived at the desired contradiction.
\end{proof}

\section{Proof of Proposition \ref{propositioncounterexample}}\label{sectioncounterexamplepropos}
Let $P$ and $G$ be as in the statement of Proposition \ref{propositioncounterexample}. It is obvious that $P$ is indeed a non-constant, homogeneous, plurisubharmonic polynomial without pluriharmonic terms and that $P$ is homogeneous of degree $2$ in $z$, $\overline{z}$, so $P$ is homogeneous in one variable separately. Furthermore, away from the coordinate hyperplanes, $P$ is pluriharmonic along the level sets of $G$ by Lemma \ref{lemmahomogsepidentifyholmap}.

{\emph{Assume for the sake of a contradiction}} that there exist a homogeneous, plurisubharmonic polynomial $Q\colon\mathbb{C}^2\to\mathbb{R}$ and holomorphic polynomials $F_1 ,F_2\colon\mathbb{C}^3\to\mathbb{C}$, homogeneous of the same degree, such that $P=Q\circ (F_1 , F_2 )$ on $\mathbb{C}^3$. As in Section \ref{prelimsection} we write
\begin{align*}
P(z,w_1 ,w_2 )= & \overline{z}\cdot{\overline{w_1}^2}\cdot\big(2z{w_1}^2 -zw_1 w_2\big)+\overline{z}\cdot\overline{w_1}\cdot\overline{w_2}\cdot\big(zw_1 w_2 -z{w_1}^2\big)\\
& +\overline{z}\cdot\overline{w_2}^2\cdot\big(z{w_2}^2\big)\text{.}
\end{align*}
Even though the holomorphic map $(F_1 , F_2 )\colon\mathbb{C}^3\to\mathbb{C}^2$ is (a priori) not necessarily non-singular, an argument analogous to the proof of Lemma \ref{lemmaplurihalonglevelsets} gives that the matrix
\begin{align*}
\begin{pmatrix}
\frac{\partial F_1}{\partial z}   & \frac{\partial F_1}{\partial w_1} & \frac{\partial F_1}{\partial w_2} \\
\frac{\partial F_2}{\partial z}   & \frac{\partial F_2}{\partial w_1} & \frac{\partial F_2}{\partial w_2} \\
2{w_1}^2-w_1 w_2 & 4zw_1 -zw_2 & -zw_1 \\
w_1 w_2 -{w_1}^2 & -2zw_1 +zw_2 & zw_1 \\
{w_2}^2 & 0 & 2zw_2
\end{pmatrix}
\end{align*}
has the same rank as the matrix
\begin{align*}
\begin{pmatrix}
\frac{\partial F_1}{\partial z}   & \frac{\partial F_1}{\partial w_1} & \frac{\partial F_1}{\partial w_2} \\
\frac{\partial F_2}{\partial z}   & \frac{\partial F_2}{\partial w_1} & \frac{\partial F_2}{\partial w_2}
\end{pmatrix}
\end{align*}
at every point of some non-empty open subset $U$ of $\mathbb{C}^3$, which does not meet the coordinate hyperplanes. In particular, said rank is $2$ and $(F_1 ,F_2 )$ is non-singular on $U$; hence $F_1$ and $F_2$ are both non-constant and homogeneous of degree $d:=\deg F_1 =\deg F_2 \geq 1$. 
After adding the fourth row (of the former matrix) to the third row and applying the identity theorem, we get that the following holds on $\mathbb{C}^3$ for $j\in\{1,2\}$:
\begin{align*}
0=\det\begin{pmatrix}
\frac{\partial F_j}{\partial z}   & \frac{\partial F_j}{\partial w_1} & \frac{\partial F_j}{\partial w_2} \\
{w_1}^2 & 2zw_1 & 0\\
{w_2}^2 & 0 & 2zw_2
\end{pmatrix}\text{,}
\end{align*}
and, using that $F_j$ is homogeneous, a calculation then gives
\begin{align*}
3z\frac{\partial F_j}{\partial z}=z\frac{\partial F_j}{\partial z}+w_1\frac{\partial F_j}{\partial w_1}+w_2\frac{\partial F_j}{\partial w_2}=d\cdot F_j\text{.}
\end{align*}
From this we readily deduce that $d/3$ is a positive integer and that there exist holomorphic polynomials $0\not\equiv f_1 ,f_2\colon\mathbb{C}^2\to\mathbb{C}$, homogeneous of degree $2d/3$, such that
\begin{align*}
F_j (z,w_1 ,w_2 )=z^{\frac{d}{3}}\cdot f_j (w_1 ,w_2 )\text{ on }\mathbb{C}^3
\end{align*}
for $j\in\{1,2\}$. Since $Q\colon\mathbb{C}^2\to\mathbb{R}$ is a homogeneous, plurisubharmonic (and clearly also non-pluriharmonic) polynomial, its degree $\deg Q$ is even. But $d/3$ is an integer and $6=\deg P=d\cdot \deg Q$, so we necessarily have $\deg Q=2$ and $d=3$. In particular, $f_j$ is homogeneous of degree $2$ and $F_j(z, w_1 ,w_2 )=z\cdot f_j (w_1 ,w_2 )$ on $\mathbb{C}^3$ for $j\in\{1,2\}$. Since $P$ does not have any pluriharmonic terms, we can assume that $Q$ does not have any pluriharmonic terms either. So there exist $a,c\in\mathbb{R}$, $b\in\mathbb{C}$, such that we have for all $(x,y)\in\mathbb{C}^2$:
\begin{align*}
Q(x,y)=a\cdot |x|^2 +b\cdot x\overline{y}+\overline{b}\cdot\overline{x}y+c\cdot |y|^2\text{.}
\end{align*}
For $j\in\{1,2\}$ we furthermore find $\sigma_j ,\rho_j ,\mu_j\in\mathbb{C}$, such that we have for all $(w_1 ,w_2 )\in\mathbb{C}^2$:
\begin{align*}
f_j (w_1 ,w_2 )=\sigma_j {w_1}^2 +\rho_j w_1 w_2 +\mu_j {w_2}^2\text{.}
\end{align*}
A calculation then shows that
\begin{align*}
& Q(z\cdot f_1 (w_1 ,w_2 ),z\cdot f_2 (w_1 ,w_2 ))\\
= & |z|^2\cdot\Big(
\overline{w_1}^2\cdot  g_1 (w_1 ,w_2 ) +\overline{w_1 w_2}\cdot  h(w_1 ,w_2 ) +\overline{w_2}^2\cdot  g_2 (w_1 ,w_2 ) 
\Big)
\end{align*}
for some $g_1 ,h,g_2\colon\mathbb{C}^2\to\mathbb{C}$ contained in the $\mathbb{C}$-vector space $\mathcal{V}$ spanned by $f_1$ and $f_2$; we trivially have $\dim_\mathbb{C} \mathcal{V}\leq 2$. Recalling that $P=Q\circ (F_1 ,F_2 )$, we necessarily have
\begin{align*}
g_1 (w_1 ,w_2 )=2{w_1}^2 -w_1 w_2\text{,} && h(w_1 ,w_2 )=w_1 w_2 -{w_1}^2\text{,} && g_2 (w_1 ,w_2 )={w_2}^2
\end{align*}
which implies that $\dim_\mathbb{C} \mathcal{V}\geq 3$. We have arrived at the desired contradiction.

\section{Proof of Theorem \ref{theoremfolisinglefct}}\label{folisinglefctsection}

In this section we will (without further comment) identify holomorphic polynomials $\mathbb{C}^n\to\mathbb{C}$ with elements of the polynomial ring $\mathbb{C}[z_1 ,\dots , z_n ]$ in the obvious way.

\theoremstyle{definition}
\newtheorem{notationmaxholroot}[propo]{Notation}
\begin{notationmaxholroot}
	\label{notationmaxholroot}
	Let $0\neq g\in\mathbb{C}[z_1 ,\dots , z_n ]$ be a homogeneous polynomial of positive degree. Then the set
	\begin{align*}
	\{m\in\mathbb{Z}_{>0}\colon g=\widetilde{g}^m\text{ for some homogeneous }\widetilde{g}\in\mathbb{C}[z_1 ,\dots ,z_n ] \}
	\end{align*}
	is clearly non-empty and bounded from above. Hence it has a maximum, which we denote as $M_g\in\mathbb{Z}_{>0}$.
\end{notationmaxholroot}

\theoremstyle{plain}
\newtheorem{lemmaalgebraic}[propo]{Lemma}
\begin{lemmaalgebraic}
	\label{lemmaalgebraic}
	Let $0\neq g\in\mathbb{C}[z_1 ,\dots , z_n ]$ be a homogeneous polynomial of positive degree and let $h\in\mathbb{C}[z_1 ,\dots , z_n ]$ be any homogeneous polynomial with $g=h^{M_g}$ (see Notation \ref{notationmaxholroot}).
	If $0\neq f\in\mathbb{C}[z_1 ,\dots , z_n ]$ is a homogeneous polynomial of positive degree satisfying
	\begin{align*}
	(\deg g)\cdot g\cdot\frac{\partial f}{\partial z_l}=(\deg f)\cdot f\cdot\frac{\partial g}{\partial z_l}\text{ for all }l\in\{1,\dots ,n\}\text{,}
	\end{align*}
	then $M_g\cdot (\deg f)/(\deg g)$ is a positive integer and there exists a $c\in\mathbb{C}\setminus\{0\}$, such that $f=c\cdot h^{M_g\cdot\frac{\deg f}{\deg g}}$.
\end{lemmaalgebraic}

\begin{proof}
Let $g,h,f\in\mathbb{C}[z_1 ,\dots ,z_n ]$ be as in the statement of the lemma. Write
\begin{align*}
g=u{p_1}^{\alpha_1}\cdots{p_m}^{\alpha_m}\text{,}
\end{align*}
where $m$ is a positive integer (since $\deg g>0$), $u\in\mathbb{C}\setminus\{0\}$ is a unit, $p_1,\dots ,p_m\in\mathbb{C}[z_1 ,\dots ,z_n ]$ are pairwise non-associate primes, and $\alpha_1 ,\dots ,\alpha_m$ are positive integers. Since $\deg f>0$, we get that $g$ divides $f\cdot (\partial g)/(\partial z_l )$ for $l=1,\dots ,n$. Hence, considering any $s\in\{1,\dots ,m\}$, we get that ${p_s}^{\alpha_s}$ divides
\begin{align*}
f\cdot\frac{\partial g}{\partial z_l}=f\cdot u\cdot \sum_{j=1}^{m}{p_1}^{\alpha_1}\cdots\widehat{{p_j}^{\alpha_j}}\cdots{p_m}^{\alpha_m}\cdot\alpha_j\cdot {p_j}^{\alpha_j -1}\cdot\frac{\partial p_j}{\partial z_l}\text{.}
\end{align*}
If $j\neq s$, then the corresponding summand is trivially divisible by ${p_s}^{\alpha_s}$. But this implies that ${p_s}^{\alpha_s}$ divides
\begin{align*}
f\cdot u\cdot{p_1}^{\alpha_1}\cdots\widehat{{p_s}^{\alpha_s}}\cdots{p_m}^{\alpha_m}\cdot\alpha_s\cdot {p_s}^{\alpha_s -1}\cdot\frac{\partial p_s}{\partial z_l}\text{.}
\end{align*}
Hence
\begin{align*}
p_s\biggm| f\cdot {p_1}^{\alpha_1}\cdots\widehat{{p_s}^{\alpha_s}}\cdots{p_m}^{\alpha_m}\cdot\frac{\partial p_s}{\partial z_l}\text{.}
\end{align*}
Since $p_s$ is prime, it divides one of the factors. Since the primes $p_1,\dots ,p_m$ are pairwise non-associate, we get for all $l\in\{1,\dots ,n\}$:
\begin{align*}
p_s\bigm| f\text{ or }p_s\biggm|\frac{\partial p_s}{\partial z_l}\text{.}
\end{align*}
But since $p_s$ is prime and hence $(\partial p_s )/(\partial z_{l_s} )\neq 0$ for some $l_s\in\{1,\dots ,n\}$, we get $p_s\bigm| f$. We have shown that every prime factor of $g$ divides $f$. But by reversing the roles of $f$ and $g$ and repeating the same argument, we also get that every prime factor of $f$ divides $g$. We conclude that there exist a unit $v\in\mathbb{C}\setminus\{0\}$ and positive integers $\beta_1 ,\dots ,\beta_m$, such that
\begin{align*}
f=v{p_1}^{\beta_1}\cdots{p_m}^{\beta_m}\text{.}
\end{align*}
By assumption we then have for all $l\in\{1,\dots ,n\}$:
\begin{align*}
& (\deg g)\cdot u{p_1}^{\alpha_1}\cdots{p_m}^{\alpha_m}\cdot v\cdot \sum_{j=1}^{m}{p_1}^{\beta_1}\cdots\widehat{{p_j}^{\beta_j}}\cdots{p_m}^{\beta_m}\cdot\beta_j\cdot {p_j}^{\beta_j -1}\cdot\frac{\partial p_j}{\partial z_l}\\
= & (\deg f)\cdot v{p_1}^{\beta_1}\cdots{p_m}^{\beta_m}\cdot u\cdot \sum_{j=1}^{m}{p_1}^{\alpha_1}\cdots\widehat{{p_j}^{\alpha_j}}\cdots{p_m}^{\alpha_m}\cdot\alpha_j\cdot {p_j}^{\alpha_j -1}\cdot\frac{\partial p_j}{\partial z_l}\text{,}
\end{align*}
and hence
\begin{align*}
0=\sum_{j=1}^{m}p_1\cdots\widehat{p_j}\cdots p_m\cdot\frac{\partial p_j}{\partial z_l}\cdot (\beta_j (\deg g)-\alpha_j (\deg f))
\text{.}
\end{align*}
Considering any $s\in\{1,\dots ,m\}$, we note that $p_s$ obviously divides the $j$-th summand for $j\neq s$, and hence
\begin{align*}
p_s\biggm| p_1\cdots\widehat{p_s}\cdots p_m\cdot\frac{\partial p_s}{\partial z_l}\cdot (\beta_s (\deg g)-\alpha_s (\deg f))\text{.}
\end{align*}
Using again that the primes $p_1,\dots ,p_m$ are pairwise non-associate and considering some $l_s\in\{1,\dots ,n\}$ with $(\partial p_s )/(\partial z_{l_s})\neq 0$, we get
\begin{align*}
p_s\bigm| (\beta_s (\deg g)-\alpha_s (\deg f))\text{,}
\end{align*}
i.e., $\beta_s (\deg g)-\alpha_s (\deg f)=0$. So, since $s$ was chosen arbitrarily, we have
\begin{align*}
\frac{\beta_j}{\alpha_j}=\frac{\deg f}{\deg g}\text{ for all }j\in\{1,\dots ,m\}\text{.}
\end{align*}
Since $g=h^{M_g}$, we can write
\begin{align*}
h=w{p_1}^{\gamma_1}\cdots{p_m}^{\gamma_m}\text{,}
\end{align*}
where $w\in\mathbb{C}\setminus\{0\}$, $w^{M_g}=u$, and $\gamma_j=\alpha_j /M_g$ is a positive integer for $j=1,\dots ,m$. Due to the defining properties of $M_g$ (see Notation \ref{notationmaxholroot}) we furthermore have $\gcd (\gamma_1 ,\dots ,\gamma_m )=1$, i.e., there exist $d_1 ,\dots ,d_m\in\mathbb{Z}$, such that
\begin{align*}
1=\sum_{j=1}^{m}d_j\gamma_j\text{.}
\end{align*}
Hence
\begin{align*}
M_g\cdot\frac{\deg f}{\deg g}=\sum_{j=1}^{m}d_j\cdot\gamma_j\cdot M_g\cdot\frac{\deg f}{\deg g}=\sum_{j=1}^{m}d_j\cdot\frac{\alpha_j}{M_g}\cdot M_g\cdot\frac{\beta_j}{\alpha_j}=\sum_{j=1}^{m}d_j\beta_j
\end{align*}
is a (positive) integer, as desired. Finally, we compute
\begin{align*}
h^{M_g\cdot\frac{\deg f}{\deg g}} & =w^{M_g\cdot\frac{\deg f}{\deg g}}\cdot\prod_{j=1}^{m}\left({p_j}^{\gamma_j}\right)^{M_g\cdot\frac{\beta_j}{\alpha_j}}\\
& =w^{M_g\cdot\frac{\deg f}{\deg g}}\cdot\prod_{j=1}^{m}\left({p_j}^{\frac{\alpha_j}{M_g}}\right)^{M_g\cdot\frac{\beta_j}{\alpha_j}}\\
& =w^{M_g\cdot\frac{\deg f}{\deg g}}{p_1}^{\beta_1}\cdots{p_m}^{\beta_m}\\
& =\frac{w^{M_g\cdot\frac{\deg f}{\deg g}}}{v}\cdot f\text{,}
\end{align*}
and the claim follows.
\end{proof}

Armed with Lemma \ref{lemmaalgebraic}, we can provide a proof for Theorem \ref{theoremfolisinglefct}.

\begin{proof}[Proof of Theorem \ref{theoremfolisinglefct}]
As in Section \ref{prelimsection} we write
\begin{align*}
P(z)=\sum_{\beta\colon 1\leq |\beta |\leq 2k-1}{\overline{z}^{\beta}P_{\beta}(z)}\text{.}
\end{align*}
We then apply Lemma \ref{lemmaplurihalonglevelsets} and get for all $\beta\in (\mathbb{Z}_{\geq 0})^n$ with $1\leq |\beta |\leq 2k-1$ and for all $j,l\in\{1,\dots ,n\}$:
\begin{align*}
\frac{\partial P_\beta}{\partial z_l}\frac{\partial G}{\partial z_j}=\frac{\partial P_\beta}{\partial z_j}\frac{\partial G}{\partial z_l}\text{ on }\mathbb{C}^n\text{.}
\end{align*}
Writing $G=G(0)+q+R$, where $q\colon\mathbb{C}^n\to\mathbb{C}$ is a non-constant, homogeneous, holomorphic polynomial and $R\colon\mathbb{C}^n\to\mathbb{C}$ is a holomorphic function whose Taylor series at $0$ does not involve any terms of degree $\leq \deg q$, we get, owing to the fact that the $P_\beta$ are homogeneous:
\begin{align*}
\frac{\partial P_\beta}{\partial z_l}\frac{\partial q}{\partial z_j}=\frac{\partial P_\beta}{\partial z_j}\frac{\partial q}{\partial z_l}\text{ on }\mathbb{C}^n
\end{align*}
for all $\beta\in (\mathbb{Z}_{\geq 0})^n$ with $1\leq |\beta |\leq 2k-1$ and for all $j,l\in\{1,\dots ,n\}$. By multiplying with $z_j$ and then summing over $j$, we get, using that both $P_\beta$ and $q$ are homogeneous:
\begin{align*}
(\deg P_\beta )\cdot P_\beta\cdot \frac{\partial q}{\partial z_l}=(\deg q)\cdot q\cdot \frac{\partial P_\beta}{\partial z_l}\text{ on }\mathbb{C}^n
\end{align*}
for all $\beta ,l$. Let $h\colon\mathbb{C}^n\to\mathbb{C}$ be any homogeneous, holomorphic polynomial with $q=h^{M_q}$ (see Notation \ref{notationmaxholroot}). Lemma \ref{lemmaalgebraic} then implies that, for a given $\beta$, the following is true:
\begin{itemize}
	\item if $M_q\cdot\frac{\deg P_\beta}{\deg q}=M_q\cdot\frac{\deg P_\beta}{M_q\cdot\deg h}=\frac{2k-|\beta |}{\deg h}$ is not a positive integer, then $P_\beta \equiv 0$,
	\item if $\frac{2k-|\beta |}{\deg h}$ is a positive integer, then there exists a $c_\beta\in\mathbb{C}$, such that
	\begin{align*}
	P_\beta =c_\beta\cdot h^{M_q\cdot\frac{\deg P_\beta}{\deg q}}=c_\beta\cdot h^{\frac{2k-|\beta |}{\deg h}}
	\end{align*}
	(if $P_\beta\equiv 0$, take $c_\beta =0$, otherwise apply Lemma \ref{lemmaalgebraic}).
\end{itemize}
Since $P\not\equiv 0$, there exists a positive integer $L$, such that
\begin{align*}
\{1,\dots ,L\}=\left\{\frac{2k-|\beta |}{\deg h}\colon\beta\in (\mathbb{Z}_{\geq 0})^n,1\leq |\beta |\leq 2k-1, \frac{2k-|\beta |}{\deg h}\in\mathbb{Z}_{>0} \right\}\text{.}
\end{align*}
Hence we can write on $\mathbb{C}^n$:
\begin{align*}
P(z) & =\sum_{l=1}^{L}\sum_{\beta\colon |\beta |=2k-l\cdot\deg h}\overline{z}^{\beta}\cdot c_\beta\cdot (h(z))^l\\
& =\sum_{l=1}^{L} (h(z))^l\cdot\sum_{\beta\colon |\beta |=2k-l\cdot\deg h}c_\beta\cdot\overline{z}^{\beta}\\
& =\sum_{l=1}^{L} \overline{h(z)}^l\cdot\sum_{\beta\colon |\beta |=2k-l\cdot\deg h}\overline{c_\beta}\cdot{z}^{\beta}\text{,}
\end{align*}
where the last equality is due to the fact that $P$ is real-valued. Since $h^l\not\equiv 0$, we can find an $\alpha_l\in (\mathbb{Z}_{\geq 0})^n$ with $|\alpha_l |=l\cdot\deg h$, such that $z^{\alpha_l}$ appears with coefficient $\gamma_l\neq 0$ in the Taylor expansion of $h^l$ at $0$ (note that $\alpha_l$ is not uniquely determined in general). Recalling that
\begin{align*}
P(z)=\sum_{\beta\colon 1\leq |\beta |\leq 2k-1}{\overline{z}^{\beta}P_{\beta}(z)}\text{,}
\end{align*}
we see that we necessarily have (recall that $h$ is homogeneous):
\begin{align*}
P_{\alpha_l}(z)=\overline{\gamma_l}\cdot \sum_{\beta\colon |\beta |=2k-l\cdot\deg h}\overline{c_\beta}\cdot{z}^{\beta}\text{,}
\end{align*}
and hence
\begin{align*}
P(z)=\sum_{l=1}^{L} \overline{h(z)}^l\cdot\frac{1}{\overline{\gamma_l}}\cdot P_{\alpha_l}(z)\text{.}
\end{align*}
Assume for the sake of a contradiction that $2k/(\deg h)$ is not an integer. Then we have for all $l=1,\dots ,L$ that
\begin{align*}
\frac{2k-|\alpha_l |}{\deg h}=\frac{2k}{\deg h}-l\not\in\mathbb{Z},
\end{align*}
which implies that $P_{\alpha_l}\equiv 0$ (see above). But then $P\equiv 0$ and we arrive at the desired contradiction. Hence $2k/(\deg h)\in\mathbb{Z}$. But then we have that $\frac{2k-|\alpha_l |}{\deg h}$ is a positive integer for $l=1,\dots ,L$, which, using the above, implies that
\begin{align*}
P_{\alpha_l}=c_{\alpha_l}\cdot h^{\frac{2k-|\alpha_l |}{\deg h}}=c_{\alpha_l}\cdot h^{\frac{2k}{\deg h}-l}\text{.}
\end{align*}
We now define $s\colon\mathbb{C}\to\mathbb{R}$,
\begin{align*}
\tau\mapsto\sum_{l=1}^{L}\frac{c_{{\alpha_l}}}{\overline{\gamma_l}}\cdot\overline{\tau}^l\cdot\tau^{\frac{2k}{\deg h}-l}\text{.}
\end{align*}
Note that $s$ is indeed real-valued. It is now easy to see that $s$ and $h$ have all the desired properties.
\end{proof}

\section{Proof of Theorem \ref{theoremhomogseparately} and Corollary \ref{corollaryhomogeveryvarseparately}}\label{sectionproofoftheoremhomogseparately}
In this section we will provide a proof for Theorem \ref{theoremhomogseparately}. We will not provide a separate proof for Corollary \ref{corollaryhomogeveryvarseparately}, since it will be obvious from the proof of Theorem \ref{theoremhomogseparately}.

\begin{proof}[Proof of Theorem \ref{theoremhomogseparately}]
We adapt the notation from Section \ref{prelimsection}. In particular we write
\begin{align*}
P(z)=\sum_{\beta\colon 1\leq |\beta |\leq 2k-1}{\overline{z}^{\beta}P_{\beta}(z)}=\sum_{\beta\in B}{\overline{z}^{\beta}P_{\beta}(z)}\text{,}
\end{align*}
where $B=\{\beta\in ({\mathbb{Z}_{\geq 0}})^n\colon 1\leq |\beta |\leq 2k-1\text{ and }P_{\beta}\not\equiv 0 \}$.
By assumption on $P$ we find, for every $\beta \in B$, a holomorphic polynomial $q_{\beta}\colon\mathbb{C}^{n-l}\to\mathbb{C}$, homogeneous of degree $2k-(2d_1 +\dots +2d_l +\beta_{l+1}+\dots +\beta_n )$, such that
\begin{align*}
P_\beta (z_1,\dots ,z_n )={z_1}^{2d_1 -\beta_1}\cdots{z_l}^{2d_l -\beta_l}q_{\beta}(z_{l+1},\dots ,z_{n})\text{.}
\end{align*}
For ease of notation we write $m=n-l$ and $(w_1 ,\dots ,w_m )=(z_{l+1},\dots ,z_{n})$. We will switch back and forth between notations whenever convenient.
By Lemma \ref{lemmahomogsepidentifyholmap}, away from the coordinate hyperplanes, $P$ is pluriharmonic along the level sets of $G\colon\mathbb{C}^n\to\mathbb{C}^{m}$,
\begin{align*}
G(z_1 ,\dots ,z_l ,w_1 ,\dots ,w_m )={z_1}^{d_1}\cdots{z_l}^{d_l}\cdot \left({w_1}^{k-D},\dots ,{w_m}^{k-D}\right)\text{.}
\end{align*}
Lemma \ref{lemmaplurihalonglevelsets} then gives that the following holds on $\mathbb{C}^n$ for all $\beta\in B$ and for all $\nu\in\{1,\dots ,n\}$:
\begin{align*}
\det
\begin{pmatrix}
\frac{\partial G_1}{\partial w_1} & \dots  & \frac{\partial G_1}{\partial w_m} \\
\vdots &  \ddots & \vdots \\
\frac{\partial G_m}{\partial w_1} & \dots  & \frac{\partial G_m}{\partial w_m}
\end{pmatrix}\cdot \frac{\partial P_{\beta}}{\partial z_\nu}
=\sum_{j=1}^{m}\det
\begin{pmatrix}
\frac{\partial G_1}{\partial w_1} & \dots  & \frac{\partial G_1}{\partial w_m} \\
\vdots &  \ddots & \vdots \\
\frac{\partial G_{j-1}}{\partial w_1} & \dots  & \frac{\partial G_{j-1}}{\partial w_m} \\
\frac{\partial P_{\beta}}{\partial w_1} & \dots  & \frac{\partial P_{\beta}}{\partial w_m} \\
\frac{\partial G_{j+1}}{\partial w_1} & \dots  & \frac{\partial G_{j+1}}{\partial w_m} \\
\vdots &  \ddots & \vdots \\
\frac{\partial G_m}{\partial w_1} & \dots  & \frac{\partial G_m}{\partial w_m}
\end{pmatrix}\cdot \frac{\partial G_j}{\partial z_\nu}\text{.}
\end{align*}
A calculation then shows that the following holds on $\mathbb{C}^n$ for all $\beta\in B$, $\nu\in\{1,\dots ,l\}$:
\begin{align*}
(k-D)(2d_\nu -\beta_\nu ) q_\beta (w_1 ,\dots ,w_m )=d_\nu \sum_{j=1}^{m} w_j\frac{\partial q_\beta}{\partial w_j}(w_1 ,\dots ,w_m )\text{,}
\end{align*}
but since $q_\beta\not\equiv 0$ is homogeneous, this simplifies to
\begin{align*}
& (k-D)(2d_\nu -\beta_\nu ) q_\beta (w_1 ,\dots ,w_m )\\
= & d_\nu \cdot (2k-(2d_1 +\dots +2d_l +\beta_{l+1}+\dots +\beta_n ))\cdot q_\beta (w_1 ,\dots ,w_m )\text{.}
\end{align*}
For $\beta\in B$ we have $q_\beta\not\equiv 0$, so we get for all $\beta\in B$, $\nu\in\{1,\dots ,l\}$:
\begin{align*}
(k-D)(2d_\nu -\beta_\nu )= & d_\nu \cdot (2k-(2d_1 +\dots +2d_l +\beta_{l+1}+\dots +\beta_n ))\\
= & d_\nu\cdot M_\beta\text{,}
\end{align*}
where $M_\beta$ is defined in the obvious way. Since $d=\gcd (d_1 ,\dots ,d_l ,k)$, there exist $c,c_1 ,\dots ,c_l\in\mathbb{Z}$, such that $d=c\cdot (k-D)+\sum_{\nu =1}^{l}c_\nu\cdot d_\nu$. But then we have for all $\beta\in B$:
\begin{align*}
\frac{d\cdot M_\beta}{k-D}=c\cdot M_\beta + \sum_{\nu =1}^{l}c_\nu d_\nu \frac{M_\beta}{k-D}=c\cdot M_\beta +\sum_{\nu =1}^{l}c_\nu (2d\nu -\beta_\nu )\text{,}
\end{align*}
so $d\cdot M_\beta /(k-D)$ is an integer for all $\beta\in B$. If $M_\beta$ was $0$ for some $\beta\in B$, then $2d_\nu -\beta_\nu$ would be $0$ for $\nu\in\{1,\dots ,l\}$, implying that $|\beta |=2k$, in contradiction to $\beta\in B$. Hence $d\cdot M_\beta /(k-D)$ is a {\emph{positive}} integer for all $\beta\in B$. This implies that for all $\beta\in B$ we have
\begin{align*}
P_\beta (z_1 ,\dots ,z_n )=\left({z_1}^{\frac{d_1}{d}}\cdots{z_l}^{\frac{d_l}{d}}\right)^{\frac{d\cdot M_\beta}{k-D}}\cdot q_\beta (z_{l+1},\dots ,z_n )\text{,}
\end{align*}
where all occurring exponents are positive integers. If $\beta\in B$, then $M_\beta\leq 2(k-D)$. If this was an equality, then, similarly to above, we would get $|\beta |=0$, in contradiction to $\beta\in B$. Hence we have for all $\beta\in B$:
\begin{align*}
\frac{d\cdot M_\beta}{k-D}, 2d-\frac{d\cdot M_\beta}{k-D}\in\{1,2,\dots ,2d-1\}\text{.}
\end{align*}
For $\beta\in B$, $\nu\in\{1,\dots ,l\}$ we have
\begin{align*}
\beta_\nu =2d_\nu -\frac{d_\nu\cdot M_\beta}{k-D}=\frac{d_\nu}{d}\cdot\left( 2d-\frac{d\cdot M_\beta}{k-D} \right)\text{,}
\end{align*}
and both factors in the latter equality are positive integers.
We now calculate, noting that all occurring exponents are positive integers:
\begin{align*}
	P(z) & =\sum_{\beta\in B}{\overline{z}^{\beta}P_{\beta}(z)}\\
	& =\sum_{\beta\in B}\left({z_1}^{\frac{d_1}{d}}\cdots{z_l}^{\frac{d_l}{d}}\right)^{\frac{d\cdot M_\beta}{k-D}}\cdot{\overline{\left({z_1}^{\frac{d_1}{d}}\cdots{z_l}^{\frac{d_l}{d}}\right)}^{2d-\frac{d\cdot M_\beta}{k-D}}}\\
	& \phantom{=\sum{\beta\in B}}\cdot \overline{z_{l+1}}^{\beta_{l+1}}\cdots\overline{z_n}^{\beta_n}q_{\beta}(z_{l+1},\dots ,z_n )\\
	& =\sum_{j=1}^{2d-1}\left({z_1}^{\frac{d_1}{d}}\cdots{z_l}^{\frac{d_l}{d}}\right)^{j}\cdot{\overline{\left({z_1}^{\frac{d_1}{d}}\cdots{z_l}^{\frac{d_l}{d}}\right)}^{2d-j}}\\
	& \phantom{=\sum{\beta\in B}}\cdot\sum_{\beta\in B\colon \frac{d\cdot M_\beta}{k-D}=j} \overline{z_{l+1}}^{\beta_{l+1}}\cdots\overline{z_n}^{\beta_n}q_{\beta}(z_{l+1},\dots ,z_n )\text{.}	
\end{align*}
Now if $\beta\in B$ and $j\in\{1,\dots ,2d-1\}$ satisfy $d\cdot M_\beta /(k-D)=j$, then every term occurring in $\overline{z_{l+1}}^{\beta_{l+1}}\cdots\overline{z_n}^{\beta_n}q_{\beta}(z_{l+1},\dots ,z_n )$ has holomorphic degree $M_\beta =j\cdot (k-D)/d$ and anti-holomorphic degree $\beta_{l+1} +\dots +\beta_n =(2d-j)\cdot (k-D)/d$. This implies that the polynomial $Q\colon\mathbb{C}^m\to\mathbb{R}$,
\begin{align*}
Q(w_1 ,\dots ,w_m )=\sum_{j=1}^{2d-1}\sum_{\beta\in B\colon \frac{d\cdot M_\beta}{k-D}=j} \overline{w_{1}}^{\beta_{1+l}}\cdots\overline{w_m}^{\beta_{m+l}}q_{\beta}(w_1 ,\dots ,w_m )\text{,}
\end{align*}
has the property that both the holomorphic and the anti-holomorphic degree of every term appearing in $Q$ are divisible by the integer $(k-D)/d$. Furthermore $Q(w_1 ,\dots ,w_m )=P(1,\dots ,1,w_1 ,\dots ,w_m )$, so $Q$ is (indeed real-valued and) plurisubharmonic. It is now easy to see that $Q$ has all the other desired properties.
\end{proof}

\bibliographystyle{amsplain}
\bibliography{refspaper6}

\end{document}